\documentclass[mic]{birkjour_t2}
\usepackage{amssymb,amsmath,amsthm,latexsym,booktabs,hyperref}

\theoremstyle{definition}
\newtheorem{definition}{Definition}[section]
\newtheorem{remark}[definition]{Remark}

\theoremstyle{plain}
\newtheorem{lemma}[definition]{Lemma}
\newtheorem{theorem}[definition]{Theorem}

\newcommand{\PP}{\mathbb{P}}

\begin{document}

\title[The $3 \times 3 \times 3$ hyperdeterminant]
{The $3 \times 3 \times 3$ hyperdeterminant as a polynomial in the fundamental invariants
for $SL_3(\mathbb{C}) \times SL_3(\mathbb{C}) \times SL_3(\mathbb{C})$}

\author[Bremner]{Murray Bremner}
\address[Bremner]{Department of Mathematics and Statistics, University of Saskatchewan, Saskatoon, SK, Canada}
\email{bremner@math.usask.ca}

\author[Hu]{Jiaxiong Hu}
\address[Hu]{Department of Mathematics, Simon Fraser University, Burnaby, BC, Canada}
\email{hujiaxiong@gmail.com}

\author[Oeding]{Luke Oeding}
\address[Oeding]{Department of Mathematics, University of California, Berkeley, CA, USA}
\curraddr{Department of Mathematics and Statistics, Auburn University, AL, USA}
\email{oeding@auburn.edu}



\thanks{Murray Bremner was supported by a Discovery Grant from NSERC.
Luke Oeding was supported by NSF RTG Award \# DMS-0943745.
The authors thank the two anonymous referees for suggestions which significantly improved the paper.}

\begin{abstract}
We briefly review previous work on the invariant theory of $3 \times 3 \times 3$ arrays.
We then recall how to generate arrays of arbitrary size $m_1 \times \cdots \times m_k$ with hyperdeterminant 0.
Our main result is an explicit formula for the $3 \times 3 \times 3$ hyperdeterminant as a polynomial
in the fundamental invariants $I_6, I_9, I_{12}$ for the action of the Lie group
$SL_3(\mathbb{C}) \times SL_3(\mathbb{C}) \times SL_3(\mathbb{C})$.
We apply our calculations to Nurmiev's classification of normal forms for $3 \times 3 \times 3$ arrays.
\end{abstract}

\maketitle

\allowdisplaybreaks


\section{Introduction}

The invariant theory of $3 \times 3 \times 3$ arrays, or trilinear forms on a 3-dimensional vector space,
has been studied by many authors;
the earliest reference we have been able to find is the work of Aronhold \cite{A} (1850).
A classification of the forms, and an application of classical invariant theory to this problem,
appears in Chanler and Thrall \cite{C,TC} (1938-39).
These trilinear forms and their invariants are also closely tied to the classical study of 
the Hesse configuration of 12 lines in the projective plane, 
which was addressed with a modern view by Artebani and Dolgachev \cite{AD09} (2009). 
See Dolgachev \cite{Dolgachev} (2012) for a comprehensive view.

The first results using the representation theory of Lie groups were obtained by Vinberg \cite{V} (1976),
who embedded the semisimple Lie algebra
$\mathfrak{sl}_3(\mathbb{C}) \oplus \mathfrak{sl}_3(\mathbb{C}) \oplus \mathfrak{sl}_3(\mathbb{C})$
into an exceptional Lie algebra of type $E_6$, and deduced that the algebra of invariants
is freely generated by homogeneous polynomials of degrees 6, 9 and 12, which we denote respectively by $I_6$, $I_9$ and $I_{12}$.
Strassen \cite[Theorem 4.6]{S} (1983) studied $n \times n \times 3$ arrays ($n$ odd),
showed that the complement of the set of arrays of maximal border rank is a hypersurface, and determined its equation.
For $n = 3$, this polynomial is (up to a scalar multiple) the fundamental invariant of degree 9, which vanishes
if and only if the array has rank $\le 4$; see Table \ref{ranktable} below.
Littelmann \cite{Littelmann} (1989) classified the irreducible representations of semisimple Lie groups
for which the algebra of invariants is free.

Gelfand et al.~(1992) introduced a general theory of hyperdeterminants
(modernizing a study initiated by Cayley \cite{Cayley} (1845) almost 150 years earlier),
and determined their degrees \cite[Corollary 3.9]{GKZ};
in particular, the hyperdeterminant $\Delta_{333}$ for $3 \times 3 \times 3$ arrays
is an invariant homogeneous polynomial of degree 36, and hence can be expressed in terms of the fundamental invariants
as follows, for some $a, b, c, d, e, f, g \in \mathbb{C}$:
  \begin{equation}
  \label{hdformula}
  \Delta_{333} = a I_6^6 + b I_6^4 I_{12} + c I_6^3 I_9^2 + d I_6^2 I_{12}^2 + e I_6 I_9^2 I_{12} + f I_9^4 + g I_{12}^3.
  \tag{HD}
  \end{equation}
(We note that $I_6$ and $I_9$ are uniquely determined up to nonzero scalar multiples, but $I_{12}$ is only determined up to
adding a scalar multiple of $I_6^2$.)

Ng \cite{Ng1,Ng2} (1995) described the orbits of
$PGL(\mathbb{C}^3) \times PGL(\mathbb{C}^3) \times PGL(\mathbb{C}^3)$ acting on the projective space
$\PP( \mathbb{C}^3 \otimes  \mathbb{C}^3 \otimes \mathbb{C}^3 )$,
derived explicit matrix representations for the singular forms,
and proved that the group quotient is a projective variety.
Nurmiev \cite{N} (2000) obtained an implicit description of all three fundamental invariants
in terms of convolutions of volume forms,
and classified the normal forms (orbit representatives).
Briand et al. \cite[eqns.~18, 52, 53]{BLTV} (2004), motivated by a problem from quantum computing,
used Cayley's $\Omega$-process from classical invariant theory to find explicit forms of the fundamental invariants
in terms of concomitants.
Duff and Ferrara \cite{DF} (2007) presented an analogy between supersymmetric black holes in 5 dimensions and
the bipartite entanglement of three qutrits from quantum information theory, where the common symmetry comes from
Vinberg's embedding
$\mathfrak{sl}_3(\mathbb{C}) \oplus \mathfrak{sl}_3(\mathbb{C}) \oplus \mathfrak{sl}_3(\mathbb{C}) \hookrightarrow E_6$.

Explicit computations of these invariants have also been attempted classically.  
For instance $I_{6}$ has been attributed to Aronhold, $I_{9}$ was computed by Strassen \cite{S} (1983) as a determinant of a certain commutator, 
and Ottaviani \cite{Ott07} (2007) gave a determinantal formula.
In theory Schl\"afli's method may be used to compute the $3\times 3 \times 3$ hyperdeterminant, see Gelfand et al.~\cite{GKZbook} (1994); 
this computation turns out to require a large amount of memory, 
but it may still be useful for evaluating the hyperdeterminant and was used by one of us \cite{Oed12} (2012).   
In general, these computations become very large very quickly; 
see Huggins et al.~\cite{HPYY} (2008) for explicit results on the $2 \times 2 \times 2 \times 2$ hyperdeterminant.

Domokos and Drensky \cite[Remark 2(ii)]{DD} (2012) computed the single defining relation of the algebra of
invariants for the action of $SL(3,\mathbb{C}) \times SL(3,\mathbb{C})$ on triples of $3 \times 3$ matrices,
and provided an alternative proof of Vinberg's result.
Allman et al.~\cite{AJRS} (2013) gave a generalization of Cayley's hyperdeterminant for $2 \times 2 \times 2$ arrays to
a covariant for the action of $SL_n(\mathbb{C}) \times SL_n(\mathbb{C}) \times SL_n(\mathbb{C})$ on
$n \times n \times n$ arrays.

Two of the present authors recently used computer algebra \cite{BH} (2013) to determine
the fundamental invariants of degrees 6, 9 and 12
in terms of orbit sums for the symmetry group $( S_3 \times S_3 \times S_3 ) \rtimes S_3$.
Each invariant is a homogeneous polynomial in the variables $x_{ijk}$ for $1 \le i, j, k \le 3$:
  \begin{itemize}
  \item
the polynomial $I_6$ of degree 6 from \cite[Table 5]{BH}, which is a linear combination of 8 symmetric orbits,
has 1152 terms;
  \item
the polynomial $I_9$ of degree 9 from \cite[Table 6]{BH}, which is the sum of 14 alternating orbits,
has 9216 terms;
  \item
the polynomial $I_{12}$ of degree 12 from \cite[Tables 7-11]{BH}, which is a linear combination of 235 symmetric orbits,
has 209061 terms.
  \end{itemize}
In this paper we determine the coefficients $a, b, c, d, e, f, g$ in equation \eqref{hdformula}.
A similar explicit computation of the L\"uroth invariant in terms of fundamental invariants
was recently achieved by Basson et al.~\cite{Basson} (2013), although their task was more difficult:
the algebra of invariants for homogeneous polynomials of degree 4 in three variables is not freely generated.
Interest in L\"uroth quartics was revived a few years ago by Ottaviani and Sernesi \cite{OS1,OS2};
see Ottaviani \cite{Ott12} (2012) for the computational aspects.

In \S \ref{gaussianelimination} we recall a method for generating arrays of hyperdeterminant 0
which applies to any format $m_1 \times \cdots \times m_k$.
In \S \ref{hdsection} we repeatedly generate pseudorandom $3 \times 3 \times 3$ arrays $( x_{ijk} )$
of hyperdeterminant 0, substitute their entries for the variables in the fundamental invariants,
and set the resulting polynomial to 0; in this way,
we obtain linear equations satisfied by the coefficients $a, b, c, d, e, f, g$ in \eqref{hdformula}.
In \S\ref{nurmiev} we apply our results to Nurmiev's classification of normal forms.
A similar computation for the so-called trifocal variety was carried out by one of us
for low degree covariants in \cite{AO} (2012).
In \S\ref{ranksection}, we test and determine which invariants vanish on which ranks.

Our Maple worksheet for the computations described in the present paper
is available as an ancillary file with the \texttt{arXiv} version.


\section{Gaussian elimination: projective version} \label{gaussianelimination}

We recall that a $k$-dimensional array $M = ( m_{i_1 i_2 \cdots i_k} )$ is said to have rank 1 if it is the outer product of nonzero vectors
$a_1, a_2, \dots, a_k$ in the sense that $m_{i_1 i_2 \cdots i_k} = a_{1,i_1} a_{2,i_2} \cdots a_{k,i_k}$.
The array $M$ has rank $r$ if $r$ is the least number such that $M$ can be written as a sum of $r$ arrays of rank 1.
The hyperdeterminant $\Delta = \Delta_{m_1+1, \dots, m_k+1}$ of a $k$-dimensional array of size $(m_1{+}1) \times \cdots \times (m_k{+}1)$
with $\sum_{i}(m_{i}{+}1) \ge 2 m_{j}+1$ for $j = 1, \dots, k$ is (by definition) the polynomial of minimal degree vanishing on the dual variety
of the arrays of rank 1.
(Limits of arrays of rank 1 still have rank $\leq 1$, so the set of rank 1 tensors is closed in projective space.)
If we assume that $\Delta$ is primitive in the sense that its coefficients are integers with no
common factors, then $\Delta$ is unique up to a sign.
Using basic multilinear algebra, we can apply a multilinear change of coordinates, so that the arrays in the dual variety
have zeros on the edges emanating from one corner and arbitrary scalars in the other positions.
With this description, we can easily generate pseudorandom arrays of hyperdeterminant 0 of any size $(m_1{+}1) \times \cdots \times (m_k{+}1)$.
To remove dependence on our choice of coordinates, we can apply pseudorandom changes of basis along each of
the $k$ directions.
However, by definition of invariant, these changes of basis will not affect the values of the fundamental invariants.
So for the purposes of the present paper, this last step will not be necessary.
In this section we review how to perform a multidimensional generalization (i.e.~to tensors) of Gaussian elimination
for matrices, using a projective point of view.

Let $A_1, \dots, A_k$ be vector spaces over $\mathbb{C}$, with $\dim A_i = m_i+1$ for $i = 1, \dots, k$.
The corresponding projective spaces $\PP(A_1), \dots, \PP(A_k)$ have dimensions $m_1, \dots, m_k$;
$\PP(A_i)$ consists of all lines through the origin in $A_i$.
For a nonzero vector $a_i \in A_i$ the corresponding element of $\PP(A_i)$ will be denoted $[ a_i ]$;
this is the line with direction vector $a_i$.
The tensor product $A_1 \otimes \cdots \otimes A_k$ has dimension $(m_1+1) \cdots (m_k+1)$,
and the projective space $\PP( A_1 \otimes \cdots \otimes A_k )$ has dimension
$(m_1+1) \cdots (m_k+1) - 1$.
An element of $A_1 \otimes \cdots \otimes A_k$ is a sum of simple tensors $a_1 \otimes \cdots \otimes a_k$.

\begin{definition}
The \textbf{Segre embedding} is the analogue of the tensor product in the setting of projective spaces:
  \[
  \PP(A_1) \times \cdots \times \PP(A_k) \hookrightarrow \PP( A_1 \otimes \cdots \otimes A_k ),
  \qquad
  ( \, [ a_1 ], \dots, [ a_k ] \, ) \mapsto [ \, a_1 \otimes \cdots \otimes a_k \, ].
  \]
\end{definition}

After a choice of basis, an array of size $(m_1+1) \times \cdots \times (m_k+1)$ can be identified with an element of
$A_1 \otimes \cdots \otimes A_k$.
We denote the image of the Segre embedding by $X$.
Consider an array of rank 1 in $X$, which by definition has the form $[ \, a_1 \otimes \cdots \otimes a_k \, ]$ for
nonzero $a_i \in A_i$.
A parametrized curve in $X$ through the point $[ \, a_1 \otimes \cdots \otimes a_k \, ]$ has the form
$[ \, a_1(t) \otimes \cdots \otimes a_k(t) \, ]$ where $a_i(0) = a_i$ for $i = 1, \dots, k$.
We differentiate with respect to $t$ and then set $t = 0$ to obtain
  \[
  \Big[ \, \sum_{i=1}^k a_1(t) \otimes \cdots \otimes a'_i(t) \otimes \cdots \otimes a_k(t) \, \Big]
  \xrightarrow{\; t = 0 \;}
  \Big[ \, \sum_{i=1}^k a_1 \otimes \cdots \otimes a'_i(0) \otimes \cdots \otimes a_k \, \Big].
  \]
The cone over the tangent plane to $X$ at the point $[ \, a_1 \otimes \cdots \otimes a_k \, ]$
is the span of these tangent vectors:
  \[
  T_{a_1 \cdots a_k}(X) = \sum_{i=1}^k a_1 \otimes \cdots \otimes A_i \otimes \cdots \otimes a_k.
  \]
We convert this sum of vector spaces into a direct sum by separating the common 1-dimensional subspace:
  \begin{equation}
  \label{directsum}
  \begin{array}{l}
  T_{a_1 \cdots a_k}(X)
  =
  \mathbb{C}(a_1 \otimes \cdots \otimes a_k)
  \; \oplus
  \\[3pt]
  \qquad\qquad\qquad
  \bigoplus_{i=1}^k
  ( a_1 \otimes \cdots \otimes A_i \otimes \cdots \otimes a_k ) / \mathbb{C}(a_1 \otimes \cdots \otimes a_k).
  \end{array}
  \end{equation}
We want to construct the projective dual variety $X^\vee$.

\begin{definition}
For an irreducible projective variety $X \subset \PP(V)$, the \textbf{projective dual variety} $X^\vee \subset \PP(V^\vee)$
is the Zariski closure of the set of all tangent hyperplanes; i.e.  hyperplanes $[H] \subset \PP(V^\vee)$ such that $T_{x}(X) \subset H$ for a smooth point $[x] \in X$.
\end{definition}

By a result of Gelfand et al.~\cite{GKZ}, the dual of the Segre variety is a hypersurface if and only if
$\sum_i (m_i{+}1) \ge 2 m_j+1$ for $j = 1, \dots, k$.
Projective duality preserves irreducibility, so in this case the dual of the Segre variety is a hypersurface
defined by a single equation $\Delta = 0$, where the homogeneous polynomial $\Delta$
is known as the hyperdeterminant of type
$(m_1{+}1) \times \cdots \times (m_k{+}1)$; again see \cite{GKZ}.

A hyperplane $[H]$ in projective space $\PP(A_1 \otimes \cdots \otimes A_k)$ can be identified with a normal line $N(H)$
in the dual projective space $\PP( A_1^\ast \otimes \cdots \otimes A_k^\ast )$.
The condition that $H$ contains the tangent plane $T_{a_1 \cdots a_k}(X)$ is equivalent to the condition that
the corresponding line $N(H)$ annihilates $T_{a_1 \cdots a_k}(X)$ in the natural pairing of a variety with its dual.
To work in terms of coordinates, we choose bases:
  \[
  \{ \; a_{i,1} = a_i, \; a_{i,2}, \; \dots, \; a_{i,m_i+1} \; \} \subset A_i
  \quad
  ( i = 1, \dots, k ).
  \]
The general element of $A_1^\ast \otimes \cdots \otimes A_k^\ast$ then has the form
  \begin{equation}
  \label{dualvector}
  \sum_{i_1=1}^{m_1+1} \cdots \sum_{i_k=1}^{m_k+1}
  \mu_{i_1 \dots i_k} \,
  a_{1,i_1}^\ast \otimes \cdots \otimes a_{k,i_k}^\ast \quad (\mu_{i_1 \dots i_k} \in \mathbb{C}).
  \end{equation}
By equation \eqref{directsum}, the general element of $T_{a_1 \cdots a_k}(X)$ has the following form where $d, e_{i,j} \in \mathbb{C}$:
  \begin{equation}
  \label{tangentvector}
  T_{a_1 \cdots a_k}(X)
  =
  d (a_1 \otimes \cdots \otimes a_k)
  +
  \sum_{i=1}^k
  \sum_{j=2}^{m_i+1}
  e_{i,j} ( a_1 \otimes \cdots \otimes a_{i,j} \otimes \cdots \otimes a_k ).
  \end{equation}
We evaluate the dual vector \eqref{dualvector} at the tangent vector \eqref{tangentvector} and obtain
  \[
  d \mu_{1 \dots 1}
  +
  \sum_{i=1}^k
  \sum_{j=2}^{m_i+1}
  e_{i,j} \mu_{1 \dots j \dots 1},
  \]
where the subscript $j$ of $\mu$ is in position $i$.
This must vanish for all $d$ and $e_{i,j}$ which gives the necessary and sufficient condition that
$\mu_{i_1 \dots i_k} = 0$ when $k{-}1$ (or all $k$) of the subscripts $i_1, \dots, i_k$ equal 1.

\begin{lemma}
Let $M = ( \mu_{i_1 \dots i_k} ) \in A_1 \otimes \cdots \otimes A_k$ be a $k$-dimensional array.
The following are equivalent:
  \begin{itemize}
  \item[(i)]
  The hyperdeterminant $\Delta$ of $M$ is 0.
  \item[(ii)]
  The projectivization of $M$ lies in the dual $X^\vee$ of the Segre variety.
  \item[(iii)]
  $M$ is conjugate by a multilinear change of basis in $GL(A_1) \times \cdots \times GL(A_k)$
  to an array in which
  $\mu_{i_1 \dots i_k} = 0$ when $k{-}1$ of the subscripts equal 1.
  \end{itemize}
\end{lemma}

\begin{remark}
In the familiar case $k = 2$ and $m_1 = m_2$, a square matrix $M$ is singular if and only if
there exist nonzero vectors $a_1 \in A_1$ and $a_2 \in A_2$ for which $a_1 M = 0$ and $M a_2 = 0$.
For $i = 1, 2$ we choose a basis of $A_i$ which has $a_i$ as its first vector.
With respect to these bases, the matrix $M$ is equivalent to a matrix with 0s in its first row and column,
and no condition on the remaining entries:
  \[
  M \sim
  \left[
  \begin{array}{cccc}
  0 & 0 & \cdots & 0 \\
  0 & \ast & \cdots & \ast \\
  \vdots & \vdots & \ddots & \vdots \\
  0 & \ast & \cdots & \ast
  \end{array}
  \right]
  \]
Thus $\det(M) = 0$ if and only if there exist invertible operators $\alpha_i \in GL(A_i)$ for $i = 1, 2$
such that $\alpha_1 M \alpha_2$ has this form.
\end{remark}


\section{The $3 \times 3 \times 3$ hyperdeterminant} \label{hdsection}

\begin{theorem} \label{maintheorem}
The $3 \times 3 \times 3$ hyperdeterminant has the following explicit form in terms of the fundamental invariants
$I_6$, $I_9$ and $I_{12}$ as given in \cite{BH}:
  \[
  \Delta_{333} = I_6^3 I_9^2 - I_6^2 I_{12}^2 + 36 \, I_6 I_9^2 I_{12} + 108 \, I_9^4 - 32 \, I_{12}^3.
  \]
\end{theorem}

\begin{proof}
Using the computer algebra system Maple, we generate a pseudorandom $3 \times 3 \times 3$ array $M$
with entries in the field with $p = 10007$ elements (the smallest prime $p > 10000$).
We then set the $(i,j,k)$ entry equal to 0 whenever at least two of the indices $i,j,k$ equal 1.
We compute the values modulo $p$ of the fundamental invariants $I_6, I_9, I_{12}$ and then of the invariant
monomials in equation \eqref{hdformula}.
If any value is 0, we repeat the process until all values are nonzero.

\begin{table}[h]
\[
\left[
\begin{array}{rrrrrrr}
3477 & 3766 & 6420 & 5472 & 4726 & 6898 & 9864 \\
2455 & 1031 & 7558 & 7819 &  121 & 7526 &  198 \\
4162 & 6933 & 2573 & 8969 & 2490 & 2877 & 4806 \\
2161 & 6223 & 4702 & 3440 & 3913 & 5827 & 5970 \\
9488 & 8055 & 5201 & 3051 &  550 & 2523 & 4823 \\
1465 & 9034 & 3812 & 4731 & 8452 & 4714 & 8190 \\
9339 & 6402 & 4716 & 3540 & 6875 & 1041 & 6281 \\
1722 &  411 & 9753 &  255 & 5832 & 3408 & 1473 \\
1387 & 1406 & 9172 & 2659 & 6556 & 5416 & 3655 \\
2261 & 1727 & 2982 & 5807 & 8474 & 8797 &   45
\end{array}
\right]
\qquad
\left[
\begin{array}{rrrrrrr}
1 & 0 & 0 & 0 & 0 & 0 &    0 \\
0 & 1 & 0 & 0 & 0 & 0 &    0 \\
0 & 0 & 1 & 0 & 0 & 0 & 7818 \\
0 & 0 & 0 & 1 & 0 & 0 & 2189 \\
0 & 0 & 0 & 0 & 1 & 0 & 1252 \\
0 & 0 & 0 & 0 & 0 & 1 & 3756
\end{array}
\right]
\]
\smallskip
\caption{Modular values of invariant monomials, and RCF}
\label{ptable}
\end{table}

We perform this procedure 10 times and obtain the first matrix in Table \ref{ptable};
this is the coefficient matrix modulo $p$ of a homogeneous linear system in the variables
$a, b, c, d, e, f, g$ in equation \eqref{hdformula}.
If we had done this computation using rational arithmetic, and had generated enough pseudorandom arrays,
the matrix would have a nullspace of dimension 1 with the hyperdeterminant as basis.
Hence using modular arithmetic, the nullity is at least 1.
In fact, we find that the nullity is exactly 1; the row canonical form (RCF) of the matrix modulo $p$ is the second matrix
in Table \ref{ptable}, and
the vector $[ 0, 0, 2189, 7818, 8755, 6251, 1 ]$ is a basis of the nullspace.
Using the Maple procedure \texttt{iratrecon} for rational reconstruction, we find that the simplest
rational numbers corresponding to these residue classes modulo $p$ are
  \begin{equation}
  \label{coef0}
  [a,b,c,d,e,f,g] =
  \left[ \;
  0, \; 0, \; -\frac{1}{32}, \; \frac{1}{32}, \; -\frac98, \; -\frac{27}{8}, \; 1
  \; \right]
  \end{equation}
We multiply by $-32$ to obtain the coefficients in the statement of this theorem.
We will confirm this calculation using rational arithmetic in Section \ref{ranksection}.
\end{proof}


\section{An application to Nurmiev's normal forms} \label{nurmiev}

Nurmiev \cite{N} classified the normal forms of $3 \times 3 \times 3$ arrays over $\mathbb{C}$ with
respect to the action of the Lie group
$G = SL_3(\mathbb{C}) \times SL_3(\mathbb{C}) \times SL_3(\mathbb{C})$, and in Nurmiev \cite{N2} described the orbit closure poset.

\begin{definition}
The $3 \times 3 \times 3$ array $M$ is called \textbf{semisimple} if its $G$-orbit is closed,
and \textbf{nilpotent} if the closure of its $G$-orbit contains the zero array.
\end{definition}

\begin{remark}
The variety where all the invariants vanish is called the nullcone \cite{Mumford}, and it coincides exactly
with the variety of nilpotent arrays.
It is immediate that tensors of rank $\le 2$ are nilpotent.
\end{remark}

Every $3 \times 3 \times 3$ array $M$ can be written uniquely in the form $M = S + N$ where
$S$ is semisimple, $N$ is nilpotent, and
$[S,N] = 0$ where the Lie bracket is taken from Vinberg's embedding of the vector space of $3 \times 3 \times 3$ arrays
into a simple Lie algebra of
type $E_6$.
Nurmiev defines three basic semisimple arrays; we use the notation $E_{ijk}$ for the array with 1 in position
$(i,j,k)$ and 0 in the other positions:
  \[
  X_1 = E_{111} + E_{222} + E_{333},
  \qquad
  X_2 = E_{123} + E_{231} + E_{312},
  \qquad
  X_3 = E_{132} + E_{213} + E_{321}.
  \]
Nurmiev shows that every semisimple array is equivalent to (i.e.~lies in the $G$-orbit of) an array of the normal form
  \[
  u = a_1 X_1 + a_2 X_2 + a_3 X_3.
  \]
For each type of orbit, non-vanishing of an invariant at a particular point is always a certainty, so evaluating at
pseudorandom points is sufficient.  To be certain of vanishing for all points on a parametrized orbit, we must check the
expression for the invariant restricted to the parametrized normal form.
Evaluating the three fundamental invariants and the hyperdeterminant on $u$ gives the following results:
\begin{equation}\label{eq:array}
\left.
\begin{array}{rll}
  I_6 &=
  a_1^6 + a_2^6 + a_3^6 - 10a_1^3a_2^3 - 10a_1^3a_3^3 - 10a_2^3a_3^3,
  \\[4pt]
  I_9 &=
  - (a_1 - a_2)(a_1 - a_3)(a_2 - a_3)
  (a_1^2 + a_1a_2 + a_2^2)(a_1^2 + a_1a_3 + a_3^2)(a_2^2 + a_2a_3 + a_3^2),
  \\[4pt]
  I_{12} &=
  a_1^3a_2^9 + a_1^9a_2^3 + a_1^9a_3^3 + a_1^3a_3^9 + a_2^3a_3^9 + a_2^9a_3^3
  - 4a_1^6a_2^6 - 4a_1^6a_3^6 - 4a_2^6a_3^6
  \\[4pt]
  &\qquad
  + 2a_1^3a_2^3a_3^6 + 2a_1^3a_2^6a_3^3 + 2a_1^6a_2^3a_3^3,
  \\[4pt]
  \Delta &=
  - 4a_1^3a_2^3a_3^3 (a_1 + a_2 + a_3)^3
  (a_1^2 + a_2^2 + a_3^2 + 2a_1a_2 - a_1a_3 - a_2a_3)^3
  \,\times
  \\[4pt]
  &\qquad
  (a_1^2 + a_2^2 + a_3^2 - a_1a_2 + 2a_1a_3 - a_2a_3)^3
  (a_1^2 + a_2^2 + a_3^2 - a_1a_2 - a_1a_3 + 2a_2a_3)^3
  \,\times
  \\[4pt]
  &\qquad
  (a_1^2 + a_2^2 + a_3^2 - a_1a_2 - a_2a_3 - a_1a_3)^3.
\end{array}
\quad\right\}
\end{equation}
We note that $I_9$ vanishes whenever two of the coefficients $a_1, a_2, a_3$ are equal,
and that $\Delta$ vanishes whenever any coefficient is 0 or the sum $a_1 + a_2 + a_3$ is 0.

According to the Nurmiev classification, the normal forms belong to five families.

\subsection{First family}
The coefficients of the semisimple part satisfy
  \[
  a_1 a_2 a_3 \ne 0,
  \qquad
  ( a_1^3 + a_2^3 + a_3^3 )^3 - ( 3 a_1 a_2 a_3 )^3 \ne 0.
  \]
The nilpotent part is 0, so the values of the invariants on this family are as in \eqref{eq:array}.

\subsection{Second family}
The coefficients of the semisimple part satisfy
  \[
  a_2 ( a_1^3 + a_2^3 ) \ne 0,
  \qquad
  a_3 = 0.
  \]
The possible nilpotent parts are
  \[
  E_{132} + E_{213}, \qquad E_{132}, \qquad 0.
  \]
It turns out that the values of the fundamental invariants and the hyperdeterminant do not depend on the nilpotent part,
and for the second family we find that:
  \begin{align*}
  &
  I_6 = a_1^6 + a_2^6 - 10a_1^3a_2^3,
  \qquad
  I_9 =  - a_1^3a_2^3(a_1 - a_2)(a_1^2 + a_1a_2 + a_2^2),
  \\
  &
  I_{12} = a_1^3a_2^3(a_1^6 + a_2^6 - 4a_1^3a_2^3),
  \qquad
  \Delta = 0.
  \end{align*}

\subsection{Third family}
The coefficients of the semisimple part satisfy
  \[
  a_1 \ne 0, \qquad a_2 = a_3 = 0.
  \]
The possible nilpotent parts are
  \begin{align*}
  &
  E_{123} + E_{132} + E_{213} + E_{231},
  \quad
  E_{123} + E_{132} + E_{213},
  \quad
  E_{123} + E_{132} + E_{231},
  \\
  &
  E_{123} + E_{132},
  \quad
  E_{123} + E_{231},
  \quad
  E_{132} + E_{213},
  \quad
  E_{123},
  \quad
  E_{132},
  \quad
  0.
  \end{align*}
We find that the values of the fundamental invariants and the hyperdeterminant do not depend on the nilpotent part;
for all normal forms in this family
we have $I_6 = a_1^6$ and $I_9 = I_{12} = \Delta = 0$.

\subsection{Fourth family}
The coefficients of the semisimple part satisfy
  \[
  a_1 = 0, \qquad a_3 = -a_2 \ne 0.
  \]
The possible nilpotent parts are
  \begin{align*}
  &
  E_{113} + E_{131} + E_{222} + E_{311},
  \quad
  E_{113} + E_{122} + E_{131} + E_{212} + E_{221} + E_{311},
  \\
  &
  E_{111} + E_{222},
  \quad
  E_{112} + E_{121} + E_{211},
  \quad
  E_{111},
  \quad
  0.
  \end{align*}
The values of the fundamental invariants and the hyperdeterminant do not depend on the nilpotent part;
for all normal forms in this family we have
  \[
  I_6 = 12 a_2^6, \qquad
  I_9 = 2 a_2^9, \qquad
  I_{12} = -6 a_2^{12}, \qquad
  \Delta = 0.
  \]

\subsection{Fifth family}
This family contains 24 nonzero nilpotent arrays (the semisimple part is 0); see Nurmiev \cite[Table 4, p.~723]{N}.
In every case all the fundamental invariants and the hyperdeterminant are 0.
To put this another way, we have verified that all nilpotent orbits consist of singular arrays
(from the point of view of hyperdeterminants).

\begin{remark}
These results on the evaluation of the invariants on the normal forms can be used to corroborate results
in the second paper of Nurmiev \cite{N2}
on the structure of the poset of the closures of nilpotent orbits.
\end{remark}


\section{Values of the invariants on arrays of known rank} \label{ranksection}
In this section we determine which invariants vanish on arrays of each possible rank.  In principle this could be deduced from our previous computations by determining the rank of each of Nurmiev's normal forms, but we prefer to argue directly.

Kruskal's theorem \cite{BH2,K1,K2} states that every $3 \times 3 \times 3$ array over $\mathbb{R}$ or $\mathbb{C}$ has rank at most 5.
To parametrize an array of rank $r = 1$ (strictly speaking, $r \le 1$) we take the outer product of three vectors:
$x_{ijk} = a_i b_j c_k$. To make the array pseudorandom, we assign pseudorandom values to each of the coordinates of each of the vectors in the outer product.
Then to generate parametrized (respectively pseudorandom)  arrays of rank $r \le 5$ we take the sum of $r$ parametrized (respectively pseudorandom) arrays of rank 1.

Let $\sigma_{r}^{o}$ denote the set of arrays of rank $\leq r$ and let $\sigma_{r}$ denote its (Zariski) closure. 
It is well known that $\sigma_{r}\subset \sigma_{r+1}$ for all $r$, and obviously $\sigma_{r}^{o}\subset \sigma_{r+1}^{o}$;
however it can be that $\sigma_{r}^{o} \subsetneq \sigma_{r}$.  The commonly used example is $x_{211}+x_{121}+x_{112}$, which is in $\sigma_{2}$ (and $\sigma_{3}^{o}$) but not in $\sigma_{2}^{o}$. We also note that since the zero-locus of an invariant is both (Zariski) closed and invariant under the group action, if an invariant vanishes at one point of an orbit, then it does so for all points in the orbit, and for all points in the orbit closure.  We use this fact implicitly in what follows.

The arrays of rank 1 form a single orbit, and we may test our invariants on one representative, such as $[1,0,0]^{\otimes 3}$. 

For $\sigma_{2}^{o}$, we may consider the representative $[1,0,0]^{\otimes 3} + [0,1,0]^{\otimes 3}$. 
All other arrays of rank 2 are in the closure of the orbit of this array.  
Our tests find that all three fundamental invariants vanish on these representatives.

Similarly, for $\sigma_{3}^{o}$ we may consider the representative $[1,0,0]^{\otimes 3} + [0,1,0]^{\otimes 3} + [0,0,1]^{\otimes 3}$. 
All other arrays of rank 3 are in the closure of the orbit of this array.  
We find that $I_{6}$ does not vanish on this representative, but $I_{9}, I_{12}$ and $\Delta$ do.
Experimental results for 10000 pseudorandom arrays with entries in $\{ -9, \dots, 9 \}$ and rank $\le 3$
showed that $I_6$ vanishes in only 188 cases, but $I_9$ and $I_{12}$ vanish in all 10000 cases.

For $\sigma_{4}^{o}$, the value of $I_{6}$ was already non-zero for rank $3$, so it must not vanish identically on rank 4 arrays. 
It was already determined by Strassen \cite{S} that $I_{9}$ vanishes on arrays of rank $\le 4$.
Experimental results for 10000 pseudorandom arrays with entries in $\{ -9, \dots, 9 \}$ and rank $\le 4$
showed that $I_6$ vanishes in only 4 cases, $I_9$ vanishes in all 10000 cases, and $I_{12}$ vanishes in 573 cases.
In fact, just one non-vanishing result is sufficient to indicate that $I_{12}$ does not vanish identically on the locus of rank 4 arrays. 
A similar test also shows that $\Delta$ does not vanish identically for rank 4 arrays.  
The non-vanishing of $I_{12}$ and $\Delta$ on $\sigma_{4}^{o}$ can also be deduced from Strassen's result 
that $\sigma_{4}$ is an irreducible hypersurface defined by $I_{9}$ and that neither $I_{12}$ nor $\Delta$ are multiples of $I_{9}$.

Strassen showed that his degree 9 invariant was non-zero for rank 5 arrays, thus proving that his invariant was non-trivial. 
The closure of rank 5 arrays is all of the ambient projective space, so no non-trivial invariant vanishes identically on all rank 5 arrays.
Experimental results for 10000 pseudorandom arrays with entries in $\{ -9, \dots, 9 \}$ and rank $\le 5$
showed that $I_6$ never vanishes, $I_9$ vanishes in 47 cases, and $I_{12}$ vanishes in 23 cases.
Again, just one non-vanishing example is sufficient to prove that the invariant is not identically zero on $\sigma_{5}^{o}$.
We summarize these results in Table \ref{ranktable}.

\begin{table}[h]
\[
\begin{array}{l|rrrrr}
       &\quad r \le 1 &\quad r \le 2 &\quad r \le 3   &\quad r \le 4   &\quad r \le 5   \\ \midrule
I_6    &\quad   0 &\quad   0 &\quad \ne 0 &\quad \ne 0 &\quad \ne 0 \\
I_9    &\quad   0 &\quad   0 &\quad     0 &\quad   0   &\quad \ne 0 \\
I_{12} &\quad   0 &\quad   0 &\quad     0 &\quad \ne 0 &\quad \ne 0 \\
\Delta     &\quad   0 &\quad   0 &\quad     0 &\quad \ne 0 &\quad \ne 0
\end{array}
\]
\smallskip
\caption{Values of invariants on arrays of rank $r$}
\label{ranktable}
\end{table}

We can use this method of pseudorandom arrays to reduce the size of the matrix entries when we perform the computation of \S\ref{hdsection}
using rational arithmetic.
There are 7 nonzero vectors of dimension 3 with entries in $\{0,1\}$, and hence $7^3 = 343$ nonzero
$3 \times 3 \times 3$ arrays of rank 1.
We use a pseudorandom number generator to choose 5 of these arrays; their sum $X = ( x_{ijk} )$ is an array of rank $\le 5$,
which has rank 5 with sufficiently high probability; in this case, all three fundamental invariants will have nonzero values on $X$.
We set $x_{ijk} = 0$ whenever at least two of $i,j,k$ equal 1, to ensure that $X$ has hyperdeterminant 0.
Following the same procedure as in \S\ref{hdsection}, we obtain the matrix in Table \ref{char0matrix}.
The vector of equation \eqref{coef0} is a basis for the nullspace, confirming Theorem \ref{maintheorem}
with rational arithmetic.

\begin{table}[h]
\small
\[
\left[
\begin{array}{rrrrrrr}
12230590464 &\! -509607936 &\! 28311552 &\! 21233664 &\! -1179648 &\! 65536 &\! -884736 \\
64000000 &\! -2560000 &\! 128000 &\! 102400 &\! -5120 &\! 256 &\! -4096 \\
2985984 &\! -124416 &\! 6912 &\! 5184 &\! -288 &\! 16 &\! -216 \\
75418890625 &\! -2427685000 &\! 17576000 &\! 78145600 &\! -565760 &\! 4096 &\! -2515456 \\
12230590464 &\! -509607936 &\! 28311552 &\! 21233664 &\! -1179648 &\! 65536 &\! -884736 \\
531441 &\! -19683 &\! 729 &\! 729 &\! -27 &\! 1 &\! -27 \\4096 &\! 512 &\! -256 &\! 64 &\! -32 &\! 16 &\! 8 \\
13841287201 &\! -449654478 &\! 4235364 &\! 14607684 &\! -137592 &\! 1296 &\! -474552 \\
2985984 &\! -124416 &\! 6912 &\! 5184 &\! -288 &\! 16 &\! -216 \\
11390625 &\! -303750 &\! -13500 &\! 8100 &\! 360 &\! 16 &\! -216
\end{array}
\right]
\]
\smallskip
\caption{Integer values of invariant monomials}
\label{char0matrix}
\end{table}


\end{document}